\documentclass{article}

\usepackage[utf8]{inputenc}
\usepackage{amsmath,amsthm,amssymb,amsfonts,amstext, mathtools,thmtools,thm-restate,pinlabel}
\usepackage{xcolor,graphicx}

\usepackage[T1]{fontenc}
\usepackage[stretch=10]{microtype}
\usepackage{mathtools}
\usepackage[hidelinks, linktoc=all]{hyperref}
\usepackage[capitalize,noabbrev]{cleveref}
\usepackage{fullpage}
\usepackage{lipsum}

\newcommand\MTkillspecial[1]{
  \bgroup
  \catcode`\&=9
  \let\\\relax%
  \scantokens{#1}%
  \egroup
}
\newcommand\DeclarePairedDelimiterMultiline[3]{
  \DeclarePairedDelimiter{#1}{#2}{#3}
  \reDeclarePairedDelimiterInnerWrapper{#1}{star}{
  \mathopen{##1\vphantom{\MTkillspecial{##2}}\kern-\nulldelimiterspace\right.}
  ##2
  \mathclose{\left.\kern-\nulldelimiterspace\vphantom{\MTkillspecial{##2}}##3}}
}

\title{Bounds for the random walk speed \\ in terms of the \teichmuller distance}

\author{Aitor Azemar\thanks{ Aitor.Azemar@glasgow.ac.uk, Mathematics \& Statistics, University of Glasgow G128QQ, UK}}
\date{}

\newtheorem{theorem}{Theorem}[section]
\newtheorem{lemma}[theorem]{Lemma}
\newtheorem{corollary}[theorem]{Corollary}
\newtheorem{proposition}[theorem]{Proposition}

\newtheorem{question}[theorem]{Question}

\DeclareMathOperator{\thick}{Thick}
\DeclareMathOperator{\teich}{Teich}
\DeclareMathOperator{\Isom}{Isom}
\newcommand{\longestcurv}{M}

\newcommand{\teichmuller}{Teichmüller }
\newcommand{\tray}[2]{{\mathcal{R}(#1;#2)}}

\newcommand{\R}{\mathbb{R}}
\newcommand{\N}{\mathbb{N}}

\newcommand{\T}{\mathcal{T}}
\newcommand{\C}{\mathbb{C}}

\DeclareMathOperator{\Drift}{Drift}
\DeclarePairedDelimiterMultiline{\abs}{\lvert}{\rvert}
\DeclarePairedDelimiterMultiline{\pare}{(}{)}
\DeclarePairedDelimiterMultiline{\norm}{\lVert}{\rVert}

\renewcommand{\H}{\mathbb{H}}

\renewcommand{\epsilon}{\varepsilon}

\DeclareMathOperator{\supp}{supp}

\newcommand{\hyp}{\text{Hyp}}

\numberwithin{equation}{section}

\begin{document}
 \maketitle

\begin{abstract}
Consider a closed surface $S$ with negative Euler characteristic, and an admissible probability measure on the fundamental group of $S$ with finite first moment with respect to some hyperbolic metric on $S$. Corresponding to each point in Teichmüller space there is an associated random walk on the hyperbolic plane. Azemar--Gadre--Gouëzel--Haettel--Lessa--Uyanik prove that the drift of this random walk is a proper function on Teichmüller space, and that this drift grows at least linearly with respect to the \teichmuller distance. In this paper we refine the result. On the one hand, by considering Jenkins-Strebel directions we show that the linear lower bound is sharp. On the other hand, we show that for Lebesgue typical \teichmuller geodesics, the drift grows exponentially. We also exhibit \teichmuller geodesics for which the growth oscillates between almost linear and exponential. Furthermore, we show that the drift is a quasiconvex function up to a multiplicative constant.
\end{abstract}

\section{Introduction}
Let $S$ be a closed oriented surface with negative Euler characteristic and let \(p\in S\) be a basepoint. Furthermore, let $\Gamma=\pi_1(S,p)$ and let $\mu$ be a probability measure on $\Gamma$ that is admissible, i.e., the semigroup generated by the support of $\mu$ is equal to $\Gamma$. For each hyperbolic metric $\rho$ on $S$ we can define the length of an element $g\in \Gamma$ as the minimal $\rho$-length of the free loops within the homotopy class of $g$, which we denote $|g|_\rho$. Assume then that $\mu$ has finite first moment with respect to $\rho$ (and hence, with respect to any other hyperbolic metric $\rho',$ see \cref{se:background} for details). Consider a random walk $Z^\mu_n=g_1\dotsm g_n$ 
where $g_i$ are i.i.d. elements of $\Gamma$ with distribution $\mu$. The \emph{drift (or speed)} of the random walk on $\Gamma$ driven by $\mu$ with respect to $\rho$ is defined as 
\[
\Drift_\mu(\rho) \coloneqq \lim_{n\to\infty}\frac{\abs{Z_n^\mu}_{\rho}}{n}
\]
The limit above exists and has constant value almost surely. That is, the drift is a well-defined property of the random walk, depending only on $\Gamma$, $\mu$ and $\rho$.

Denote as $\T(S)$ the \teichmuller space of $S$. It has been proven by Azemar--Gadre--Gouëzel--Haettel--Lessa--Uyanik \cite[Theorem B]{DriftProper} that the function $\Drift:\T(S)\to \R_+$ grows at least linearly with respect to the \teichmuller distance. In this paper we give more details about this relation. 

There is a natural identification between the set $Q^1(o)$ of unit area quadratic differentials $q$ under the conformal structure defined by a basepoint $o\in \T(S)$ and the set of \teichmuller geodesic rays starting at $o$. For a given $q\in Q^1(o)$ we shall denote by $\tray{q}{\cdot}:[0,\infty)\to \T(S)$ the associated geodesic. We find that whenever the quadratic differential is Jenkins--Strebel the linear lower bound of the growth is sharp. More precisely, we show the following.
\begin{theorem}\label{th:linearalongJenkin}
    Let $\mu$ be an admissible measure on $\pi_1(S,p)$ with finite first moment. Furthermore, let $o\in \T(S)$ be a basepoint, and let $q\in Q^1(o)$ be Jenkins--Strebel. Then, there are constants $c(o),C(q,o)>0$ such that, along the geodesic ray $\tray{q}{\cdot}$,
    \[
    c(o)t\le \Drift_\mu(\tray{q}{t})\le C(q,o) t
    \]
    for all $t\ge 1$.
\end{theorem}
Furthermore, we find that this is not the typical behaviour. As long as the measure on the set of directions satisfies a thickness property, we get that along typical geodesics the drift grows exponentially. For example, we find the following.
\begin{theorem}\label{th:exponentialGrowthalmosteverywhere}
     Let $\mu$ be an admissible measure on $\pi_1(S,p)$ with finite first moment. Furthermore, let $o\in \T(S)$ be a basepoint, $\theta<1$ and $\lambda$ be the Lebesgue measure on $Q^1(o)$. Then, for $\lambda$-almost all directions $q\in Q^1(o)$ we have 
    \[
     e^{\theta t}< \Drift_\mu(\tray{q}{t})<  e^{\frac{1}{\theta}t},
    \]
    for any $t>t(q)$, where $t(q)$ is some finite time depending on $q$.
\end{theorem}
See \cref{th:exponentialgrowthprecise} for a more general version of the previous result, which also applies to harmonic measures on $Q^1(o)$ arising from random walks on the mapping class group with finite first moments. We also find that these two behaviours are not the only ones, and in fact there are many intermediate behaviours. Precisely, we find the following.
\begin{theorem}\label{th:undefinedGrowth}
Let $o\in \T(S)$ be a fixed basepoint. There is a constant $c(o)>0$ such that, for any increasing diverging function $f:\R_+\to \R_+$ there is some $q\in Q^1(o)$ and diverging increasing sequences $(t_n)$, $(s_n)$ such that
\[
\Drift_\mu(\tray{q}{t_n})>c(o)e^{t_n}
\]
and
\[
\Drift_\mu(\tray{q}{s_n})< f(s_n)s_n.
\]
\end{theorem}

That is, there are directions along which the growth of the drift is exponential along one subsequence, while almost linear along another.

We also show that the quasiconvexity of hyperbolic lengths along \teichmuller geodesics proven by Lenzhen--Rafi \cite{LenzhenKasraConvexity} can be adapted to the drift.
\begin{corollary}\label{quasiconvexitycorollary}
     Let $\mu$ be an admissible measure on $\pi_1(S,p)$ with finite first moment. There exists a constant $K'>0$ such that for any \teichmuller geodesic $G$ and points $x,y,z\in \T(S)$ appearing in that order along $G$, we have
    \[
        \Drift_\mu(y)\le K' \max(\Drift_\mu(x),\Drift_\mu(z)).
    \]
\end{corollary}

The drift of the random walk is closely related with the dimension of the hitting measure. Let $\H$ be the hyperbolic plane. Each metric $\rho$ can be interpreted as a representation $\rho:\pi_1(S,p)\to \Isom_+(\H)$, which can then be extended to an action of $\pi_1(S,p)$ on the Gromov boundary $\partial \H$. We then define the measure $\mu_\rho$ on $\Isom_+(\H)$ as the pushforward measure of $\mu$ by $\rho$. A measure $\nu\in \partial \H$ is $\mu_\rho$-stationary if for every measurable $A\subset \partial \H$ we have
\[
\nu(A)= \sum_{g\in \rho(\pi_1(S,p))} \mu(\rho^{-1}(g)) \nu(g^{-1}A).
\]
The measure $\nu$ can also be interpreted as the hitting measure of a random walk on $\H$ driven by $\mu_\rho$. Each $\mu_\rho$ has a unique stationary measure, which we denote by $\nu_\rho$. Let $\dim(\nu_\rho)\in[0,1]$ denote the Hausdorff dimension of $\nu_\rho$. Deroin--Kleptsyn--Navas \cite[Conjecture 1.21]{dkn2009}, and Karlsson--Ledrappier \cite{kl2011} more generally, conjectured that the dimension is bounded away from 1 whenever the support of $\mu$ is finite. Kosenko--Tiozzo \cite{tk2020} and Azemar--Gadre--Gouëzel--Haettel--Lessa--Uyani \cite{DriftProper} recently made some partial progress regarding this conjecture, but it is still open.

Tanaka \cite{t2019} shows that
\[
\dim(\nu_\rho)=\frac{h}{\Drift_\mu(\rho)},
\]
where $h$ is the entropy of the random walk, a value depending only on the measure $\mu$. In view of this relation, each result regarding the value of the drift can be translated to a result regarding the dimension of the stationary measure. For example, from \cref{th:exponentialGrowthalmosteverywhere} we have the following.
\begin{corollary}
     Let $\mu$ be an admissible measure on $\pi_1(S,p)$ with finite first moment and finite entropy. Furthermore, let $o\in \T(S)$ be a basepoint, $0<\theta<1$ and $\lambda$ be the Lebesgue measure on $Q^1(o)$. Then, for $\lambda$-almost all directions $q\in Q^1(o)$ we have 
    \[
     e^{-\frac{1}{\theta}t}< \dim(\nu_{\tray{q}{t}})<  e^{-\theta t},
    \]
    for any $t>t(q)$, where $t(q)$ is some finite time depending on $q$.
\end{corollary}

\subsection{Proof overview}

Azemar--Gadre--Gouëzel--Haettel--Lessa--Uyani \cite{DriftProper} proved that the drift of the random walk is, up to a multiplicative constant, greater than the maximum length of a fixed finite filling set of curves. The driving observation of this paper is that this inequality can be improved to an equality up to a constant multiple. That is, we first prove the following result.
\begin{theorem}\label{scalingconstantboththeorem}
     Let $\mu$ be an admissible measure on $\pi_1(S,p)$ with finite first moment. Furthermore, let $F$ be a finite filling set of curves in $S$ and let $M_\rho^F$ be the maximum hyperbolic length among the curves with respect to the hyperbolic metric $\rho$. Then there is some $K<\infty$ depending solely on $\mu$ and $F$ such that
    \[
        \frac{1}{K} M_\rho^F \le \Drift_\mu(\rho)\le K M_\rho^F.
    \]
\end{theorem}
Many properties already known about the behaviour of hyperbolic lengths can then be translated directly to the behaviour of the drift. Of particular relevance are the works of Masur \cite{MasurTwo}, Choi--Rafi \cite{ChoiRafiComparison} and Lenzhen--Rafi \cite{LenzhenKasraConvexity}, who, respectively, find values for the growth of the hyperbolic lengths along Jenkins--Strebel directions, show relations for the length of hyperbolic curves between thick parts of \teichmuller space, and prove the quasiconvexity of the hyperbolic length up to a constant. 

Choi--Rafi's estimates of the hyperbolic lengths in the thick part, combined with Lenzhen--Rafi's quasi-convexity can be used to obtain exponential growth along the whole geodesic, provided the said geodesic does not spend too much proportional continuous time outside the thick part, relative to the total time from the start. The control on the continuous time spent outside the thick part by a common geodesic is given by estimates of Dowdall--Duchin--Masur \cite{DowdallStatisticalHyp} in the case of many Lebesgue-like measures, and of Azemar--Gadre--Jeffreys \cite{AzemarStatisticalHyp} for harmonic measures arising from random walks on the mapping class group with finite first moments.

Finally, we use that both Jenkins--Strebel and recurrent directions are dense to build a geodesic which alternates between approaching the linear growth of Jenkins--Strebel directions and the exponential growth of recurrent directions.

\subsection*{Acknowledgments}
The author would like to thank Maxime Fortier Bourque and Vaibhav
Gadre for many helpful discussions and corrections.

\section{Background}\label{se:background}

\subsection{\teichmuller space}
In this section we fix the notation and introduce some concepts needed for the rest of the paper. A detailed introduction to \teichmuller theory can be found, among others places, in the work of Farb--Margalit \cite{primer} or Gardiner \cite{GardinerBook}. Let $S$ be a closed surface with negative Euler characteristic. The \emph{\teichmuller space} of $S$ is defined as the set of equivalence classes of pairs $(X,f)$ where $X$ is a Riemann surface and $f:S\to X$ is an orientation preserving homeomorphism. Two pairs $(X,f)$ and $(Y,g)$ are equivalent if there is a conformal diffeomorphism $h:X\to Y$ such that $g^{-1}\circ h \circ f$ is isotopic to the identity.

The \emph{\teichmuller distance} between two points $[(X,f)],[(Y,g)]\in \T(S)$ is defined by $d_{\teich}([(X,f)],[(Y,g)])=\frac{1}{2}\log \inf K$, where the infimum is taken over all $K\ge 1$ such that there exists a $K$-quasiconformal homeomorphism $h: X\to Y$ with $g^{-1}\circ h \circ f$ isotopic to the identity. There is a quasiconformal map realizing the infimum.

Let $\gamma$ be an isotopy class of an essential curve in $S$ and $[(X,f)]\in \T(S)$. We shall usually refer to curves when we actually mean their isotopy classes. By the uniformization theorem there is a unique hyperbolic metric on $X$. We define the hyperbolic length of $\gamma$ in $[(X,f)]$ as the shortest length of the isotopy class $f(\gamma)$ in the unique hyperbolic metric on $X$. We shall denote this value as $\hyp_{[(X,f)]}(\gamma)$.

Given an $\epsilon>0$ the \emph{$\epsilon$-thick} part of \teichmuller space $\T_\epsilon(S)\subset \T(S)$ is the subset of marked hyperbolic metrics where all essential closed curves have hyperbolic length at least $\epsilon$.

Let $TX$ be the tangent bundle over $X$. A \emph{quadratic differential} of a Riemann surface $X$ is a map $q:TX\to \C$ such that $q(\lambda v)= \lambda^2 q(v)$ for every $\lambda\in \C$ and $v\in TX$. We denote by $Q([(X,f)])$ the holomorphic quadratic differentials with finite area under the conformal structure of $[(X,f)]$. There is a natural identification between $Q([(X,f)])$ and the cotangent space at $[(X,f)]$, so for each $q\in Q([(X,f)])$ we denote as $\tray{q}{\cdot}:\R_+\to \T(S)$ the geodesic associated to $q$. If we restrict to quadratic differentials of area $1$, which we denote as $Q^1([(X,f)])$, the associated geodesics have speed $1$. That is, for $q\in Q^1([(X,f)])$, $d_{\teich}(\tray{q}{t},\tray{q}{s})=|t-s|$. Given a quadratic differential $q$ we can define its vertical foliation $V(q)$ as the measured foliation defined by the smooth paths $\gamma$ such that $q(\gamma'(t))<0$ for every $t$ in the interior of the domain, as well as the transverse measure given by $|\text{Re}\sqrt{q}|$. If each non critical trajectory is a closed curve the quadratic differential is called \emph{Jenkins--Strebel}. The set of Jenkins--Strebel quadratic differentials is dense within the set of differentials.

By considering the sphere of radius one around the point $o$ we get an embedding of $Q^1(o)$ into $\T(S)$, so we can endow $Q^1(o)$ with a Lebesgue measure. On the other hand, Kaimonovich--Masur \cite[Theorem 2.2.4]{KaiMasur} proved that for non-elementary random walks on the mapping class group, typical sample paths converge to the Tursthon boudnary of \teichmuller space and that resulting stationary measure is supported on uniquely ergodic measured geodesics. As proven by Masur \cite{MasurTwo} differentials in $Q^1(o)$ with uniquely ergodic vertical foliations converge to the projective class of that foliation in the Thurston boundary, Hence, the stationary measure can be pulled back to get another measure on $Q^1(o)$.

\subsection{Random walks}

Let $\Gamma$ be a group and $\mu$ a probability measure on $\Gamma$. We say that $\mu$ is \emph{admissible} if the semi group generated by its support is equal to $\Gamma$. We shall always assume that $\mu$ is admissible. The random walk on $\Gamma$ determined by $\mu$ is the sequence of random variables $(Z_n)_{n\in \N}$ defined by
\[
Z_n=g_1g_2\ldots g_n,
\]
where $g_i$ are $i.i.d.$ elements of $\Gamma$ with distribution $\mu$.

Let $(\H,d)$ be the hyperbolic plane, and $G=\Isom_+(\H)$ be the group of orientation preserving isometries of $\H$. Furthermore, let $\Gamma=\pi_1(S,p)$ be the fundamental group of $S$ based at $p$ and let $\rho:\Gamma\to G$ be a discrete faithful representation. The length of an element $g\in \Gamma$ with respect to $\rho$ is defined $|g|_\rho=\inf_{x\in \H} d(x,\rho(g)x)$. Equivalently, $|g|_\rho$ is the length of the core curve of the hyperbolic cylinder $\H/\rho(g)$, or the minimal length within the isotopy class of curves associated to $\rho(g)$ in $\H/\rho(\Gamma)$. We say $\mu$ has \emph{finite first moment} with respect to $\rho$ if $\sum_{g\in \Gamma} |g|_\rho \mu(g)<\infty$. We shall always assume that $\mu$ has finite first moment with respect to some $\rho$. The \emph{linear drift} of the random walk is then defined as
\[
\Drift_\mu(\rho) := \lim_{n\to\infty} \frac{|Z_n|_\rho}{n}.
\]
By Kingman's subadditive ergodic theorem the drift is well-defined, and since the measure has finite first moment with respect to $\rho$, it is finite.

For each point $[(X,f)]\in \T(S)$ we can choose a faithful representation $\rho:\Gamma\to G$ such that $\H/\rho(\Gamma)=X$ and $\rho(\gamma)$ is the deck transformation associated to $f(\gamma)$. While such assignment is not unique, the length of an element $\gamma\in\Gamma$ is a well defined function of this assignment, and hence so is $\Drift_\mu(\rho)$. Indeed, for suitably chosen representation the curve associated to $\rho(\gamma)$ is $f(\gamma)$, so $|\gamma|_\rho=\hyp_{[(X,f)]}(\gamma)$.
Therefore the elements of $\T(S)$ can be seen as classes of discrete faithful representations $\rho:\Gamma \to G$. By a slight abuse of notation we shall treat the elements of $\T(S)$ as any choice from each class, as the properties we shall deal with are invariant within the class.

Let $F$ be a fixed finite filling set of closed curves on $S$. Even though we are asking that the support of $\mu$ generates $\Gamma$ as a semigroup, it might be that $F$ is not contained entirely in the support of $\mu^k$ for any $k\ge 1$, as the elements of the support of $\mu^k$ are formed by concatenating precisely $k$ elements from the support of $\mu$. To avoid this we add the neutral element to the support of $\mu$ as follows. Let $\epsilon>0$ and let $\mu_\epsilon=(1-\epsilon)\mu+\epsilon\delta_e$ be a relaxation of $\mu$, where we introduce a slight probability of the random walk not moving at each step. For any $\rho\in \T(S)$ we have $\Drift_{\mu_\epsilon}(\rho)=(1-\epsilon)\Drift_\mu(\rho)$. Furthermore, since $\mu$ is admissible, there is some $k\in \N$ such that the curves associated to the group elements of $\supp(\mu_{\epsilon}^k)$ contain $F$. Then, $\Drift_{\mu_\epsilon^k}(\rho)=k(1-\epsilon)\Drift_{\mu}(\rho)$. Since all the results in this paper regarding the drift are true up to multiplicative constants we will assume that $\supp(\mu)$ already contains $F$. Denote $M_\rho^F=\max_{\gamma\in F}(\hyp_\rho(\gamma))$. 
We shall first prove the following well known result. The proof is similar to the one done by Minsky \cite[Lemma 4.7]{Minsky}
\begin{lemma}\label{le:upperboundlength}
Let $F$ be a finite filling set of curves in $S$ and let $o\in \T(S)$ be a fixed basepoint. Then, there is a constant $K$ such that, for any $\rho\in \T(S)$ and free curve $\gamma$ in $S$ we have,
\[\hyp_\rho(\gamma)\le K M_\rho^F \hyp_o(\gamma).\]
\end{lemma}
\begin{proof}
Let $D_\rho$ be the maximum diameter of the polygons formed by the distance minimizing configuration of the curves $F$, and $\partial D_\rho$ its perimeter. Since the curves of $F$ cut $S$ into polygons, the value of $\hyp_\rho(\gamma)$ is bounded by the number of intersections of the curve $\gamma$ with $F$ multiplied by the maximum diameter of the polygons. That is,
\[
\hyp_\rho(\gamma)\le D_\rho \sum_{\alpha\in F}i(\gamma,\alpha)\le \partial D_\rho \sum_{\alpha\in F}i(\gamma,\alpha).
\]

Under the metric defined by $o$, for any closed curve $\alpha$ we can isometrically embed an annulus around $\alpha$ of thickness $\delta_o(\alpha)>0$. Then, for any other curve $\gamma$ we have $i(\alpha,\gamma)\delta_o(\alpha)\le \hyp_o(\gamma)$. Hence,
\[
\hyp_\rho(\gamma)\le \partial D_\rho \sum_{\alpha\in F}\frac{1}{\delta(\alpha)}\hyp_o(\gamma).
\]
Finally, the maximum perimeter of the polygons is smaller than twice the sum of the lengths of all the curves in $F$, so, denoting by $|F|$ the cardinality of $F$,
\[
\partial D_p \le 2 \sum_{\alpha\in F} \hyp_p(\alpha)\le \longestcurv^F_p 2|F|.
\]
We get the result by setting $K=2|F|\sum_{\alpha\in F}\frac{1}{\delta(\alpha)}$.
\end{proof}

Since the number of curves in $F$ is finite, $\longestcurv_\rho^F$ is finite for all $\rho\in \T(S)$. Hence, for any other $\rho\in \T(S)$ we have 
\[
    \sum_{g\in \Gamma }|g|_\rho \mu(g) \le K \longestcurv_\rho^F \sum_{g\in \Gamma} |g|_o \mu(g) <\infty.
\]
Therefore, if $\mu$ has finite first moment with respect so some basepoint $o\in \T(S)$, it has finite first moment for all $\rho\in \T(S)$. That is, it makes sense to say that the measure $\mu$ has finite first moment if it has finite first moment with respect to at least one point in $\T(S)$.

We have the following lower bound.
\begin{theorem}[Azemar--Gadre--Gouëzel--Haettel--Lessa--Uyani {\cite[Theorem 3.3]{DriftProper}}]\label{scalingconstanttheorem}
 Let $\mu$ be an admissible measure on $\pi_1(S,p)$ with finite first moment, and let $F$ be a finite filling set of curves in $S$. There exists a constant \(c > 0\) such that \(\Drift_\mu(\rho) \ge c \longestcurv_{\rho}^F\) for all \([\rho] \in \T(S)\).
\end{theorem}
Using \cref{le:upperboundlength} it is relatively straightforward to check that the value $\longestcurv_\rho^F$ also serves to give an upper bound.

\begin{proof}[Proof of \cref{scalingconstantboththeorem}]
By \cref{scalingconstanttheorem} there is $c>0$ such that $\Drift_\mu(\rho)\ge c \longestcurv_\rho^F$
, so we only have to prove the upper bound.
Fix some basepoint $o\in \T(S)$. By \cref{le:upperboundlength} there is some $C<\infty$ such that  $|Z_{n}|_\rho\le C \longestcurv_p |Z_{n}|_o$. Therefore,
\[
\Drift_\mu(\rho)=\lim_{n\to \infty} \frac{|Z_{n}|_\rho}{n}\le C \longestcurv_p\lim_{n\to \infty} \frac{|Z_{n}|_o}{n} = C \longestcurv_p \Drift_\mu(o),
\]
so the result follows from setting $K=\max\left(C\Drift_\mu(o), \frac{1}{c}\right)$
\end{proof}
\section{Linear growth}
We recall that the growth has to be at least linear.
\begin{theorem}[Azemar--Gadre--Gouëzel--Haettel--Lessa--Uyani]\label{pr:lowerboundlinear}
Given a basepoint $o\in \T(S)$ there exists some constant $c$ such that $\Drift_\mu(\rho)\ge c d_{\teich}(\rho,o)$.
\end{theorem}

In this section we prove that this bound is sharp. The main ingredient is the following result, established by Masur \cite[End of the proof of Theorem 1.1]{MasurTwo}, which finds limiting values for hyperbolic lengths along Jenkins--Strebel quadratic differentials.
\begin{theorem}[Masur]\label{th:MasurStrebelBound}
    Let $q$ be a unit area Jenkins--Strebel quadratic differential and let $\alpha_1,\ldots,\alpha_k$ be its core curves. Then, for any sequence $(\rho_n)\subset \T(S)$ converging to $q$ in the visual compactification and any curve $\gamma$ in $S$ we have
    \[
        \lim_{n\to \infty} \frac{\hyp_{\rho_n}(\gamma)}{4d_{\teich}(o,\rho_n)} = \sum_{i=1}^k i(\alpha_i,\gamma).
    \]
\end{theorem}

\begin{proof}[Proof of \cref{th:linearalongJenkin}]
By \cref{pr:lowerboundlinear} we only have to prove the upper bound.
Let $\alpha_1,\ldots, \alpha_k$ be the core curves of the vertical foliation of $q$ and let $\gamma_1,\ldots,\gamma_l$ be the curves within the filling set $F$. Let $\delta>0$. Then, by \cref{th:MasurStrebelBound} for each $\gamma_j$ we have a $t_j$ such that, for all $t>t_j$, we have 
\[\hyp_{\tray{q}{t}}(\gamma_j)<4\left(\sum_{i=1}^k i(\alpha_i,\gamma_j)+\delta\right)t.\] 
Hence we can take $C(q,j)$ big enough so $\hyp_{\tray{q}{t}}(\gamma_j)< C(q,j) t$ for all $t\ge 1$. The theorem follows by setting $C(q)=K\max_{j\le k}C(q,j)$, where $K$ is the constant given by \cref{scalingconstantboththeorem}.
\end{proof}

\section{Exponential growth}
The goal of this section is to prove that the standard behaviour of the drift is exponential growth with respect to the \teichmuller distance. We begin by observing that as a direct result of Wolpert's Lemma, the growth can not be higher than exponential.
\begin{proposition}\label{pr:upperuniformbound}
Given a basepoint $o\in \T(S)$, there exists some constant $C(o)$ such that $\Drift_\mu(\rho)\le \Drift_\mu(o)e^{2d_{\teich}(\rho,o)}$.
\end{proposition}
\begin{proof}
    As proven by Wolpert \cite[Lema 3.1]{WolpertLemma} for any two points $o,\rho \in \T(S)$ and loop $\gamma$ we have
    \[
        \hyp_\rho(\gamma)\le e^{2d_{\teich}(o,\rho)}\hyp_o(\gamma).
    \]
    Therefore,
    \[
        \Drift_\mu(\rho)= \lim_{n\to \infty} \frac{\hyp_\rho(Z_n)}{n}\le e^{2d_{\teich}(o,\rho)}\lim_{n\to \infty} \frac{\hyp_o(Z_n)}{n}= e^{2d_{\teich}(o,\rho)}\Drift_\mu(o).
    \]
\end{proof}

The following result by Choi--Rafi \cite[Theorem B]{ChoiRafiComparison} allows us to improve the previous upper bound, as well as get a lower bound for the growth of the drift for points in the thick part of \teichmuller space. Recall that $\T_\epsilon(S)$ denotes the $\epsilon$-thick part of \teichmuller space.
 \begin{theorem}[Choi--Rafi]\label{choikasrauniformbound}
Fix $\epsilon>0$ and $o\in \T_\epsilon(S)$. There is a finite filling set of closed curves $G$ and a constant $D>0$ such that, for any $\rho\in T_\epsilon(S)$ we have
\[
\left|d_{\teich}(\rho,o)-\log\left(\max_{\alpha \in G} \frac{\hyp_\rho(\alpha)}{\hyp_o(\alpha)}\right)\right|\le D.
\]
 \end{theorem}

\begin{proposition} \label{pr:uniformboundthickpart}
Let $\epsilon>0$ and let $o\in \T_\epsilon(S)$. Then, there exists constants $c,C>0$ such that for any $\rho\in \T_\epsilon(S)$ we have 
\[
c e^{d_{\teich}(\rho,o)} \le \Drift_\mu(\rho) \le C e^{d_{\teich}(\rho,o)}.
\]
\end{proposition}
\begin{proof}
Let $G$ be the filling set of curves from \cref{choikasrauniformbound}. For any curve $\gamma\in G$ we have
\[
\hyp_\rho(\gamma)=\hyp_o(\gamma) \frac{\hyp_\rho(\gamma)}{\hyp_o(\gamma)} \le  \hyp_o(\gamma)\max_{\alpha \in G} \frac{\hyp_\rho(\alpha)}{\hyp_o(\alpha)}.
\]
Then, taking the maximum among all $\gamma \in G$ in the previous inequality we have, by \cref{scalingconstantboththeorem}, some $K>0$ such that
\[
\Drift_\mu(\rho)\le K \max_{\alpha\in G}{\hyp_\rho(\alpha)}\le
K\max_{\alpha\in G}\hyp_o(\alpha) e^{D} e^{d_{\teich}(\rho,o)},
\]
where the last inequality follows from applying \cref{choikasrauniformbound}. On the other hand,
\[
\frac{\max_{\alpha\in G}{\hyp_\rho(\alpha)}}{\min_{\alpha\in G}\hyp_o(\alpha) } \ge \max_{\alpha \in G} \frac{\hyp_\rho(\alpha)}{\hyp_o(\alpha)},
\]
so similarly we have
\[
\Drift_\mu(\rho)\ge \frac{1}{K} \max_{\alpha\in G}{\hyp_\rho(\alpha)}\ge
\frac{1}{K}\min_{\alpha\in G}\hyp_o(\alpha) e^{-D} e^{d_{\teich}(\rho,o)}.
\]
Hence, the result follows from setting $c=\frac{1}{K}\min_{\alpha\in G}(\hyp_o(\alpha)) e^{-D}$ and $C=K\max_{\alpha\in G}(\hyp_o(\alpha)) e^{D}$.
\end{proof}
By Mumford's compactness the preimage of every bounded subset of moduli space is contained in some $\epsilon$-thick part of \teichmuller space for $\epsilon$ small enough. Hence we have the following result.
\begin{corollary}
Let $\gamma:[0,\infty)\to \T(S)$ be a \teichmuller ray such that its image is bounded in moduli space. Then, there is $c,C>0$ such that
\[
c e^t \le \Drift_\mu(\gamma(t))\le C e^t.
\]
\end{corollary}

To prove that for almost all directions the growth is exponential we first observe that the growth of the drift is, up to a constant, not lost for big enough times. The main ingredient we shall use in the proof is the following.
\begin{theorem}[Lenzhen--Rafi {\cite[Theorem A]{LenzhenKasraConvexity}}]\label{quasiconvexityLenzhen}
    There exists a constant $K>0$ such that for every closed curve $\gamma$, any \teichmuller geodesic $G$ and points $x,y,z\in \T(S)$ appearing in that order along $G$, we have
    \[
        \hyp_y(\gamma)\le K \max(\hyp_x(\gamma),\hyp_z(\gamma)).
    \]
\end{theorem}
By \cref{scalingconstantboththeorem} Lenzhen--Rafi's result translates directly to the drift. That is, we have \cref{quasiconvexitycorollary}.
\begin{proof}[Proof of \cref{quasiconvexitycorollary}]
    By \cref{scalingconstantboththeorem,quasiconvexityLenzhen} there are constants $C,K>0$ such that
    \[
        \Drift_\mu(y)\le C \max_{\gamma \in F} \hyp_y(\gamma) \le KC \max_{\gamma \in F} \left(\max(\hyp_x(\gamma),\hyp_z(\gamma)\right)
    \]
    Switching the order of the maximums we have   
    \[ 
       \Drift_\mu(y)\le KC \max\left(\max_{\gamma \in F} \hyp_x(\gamma), \max_{\gamma \in F} \hyp_z(\gamma)\right)\le KC^2 \max(\Drift_\mu(x),\Drift_\mu(z)).
    \]
\end{proof}

 Furthermore, the drift is a proper function, so along any \teichmuller ray $\tray{q}{\cdot}$, the value of $\Drift_\mu(\tray{q}{t})$ is eventually larger than $\Drift_\mu(\tray{q}{0}$. Hence, we have the following.
 \begin{lemma}\label{le:driftgrowth}
    There is some constant $K>0$ such that, for any basepoint $o\in \T(S)$ there is some time $t_o$ for which
    \[
        \Drift_\mu(\tray{q}{s}) \le K\Drift_\mu(\tray{q}{t})
    \]
    for any $t>t_o$, $t>s>0$ and $q\in Q^1(o)$.
 \end{lemma}
\begin{proof}
    Let $t_o$ be big enough so $\Drift_\mu(\rho)>\Drift_\mu(o)$ for any $\rho$ such that $d_{\teich}(\rho,o)>t_o$. By \cref{pr:lowerboundlinear} such a $t_o$ exists. Then, for any quadratic differential $q$ based at $o$, $t>t_o$ and $t>s>0$ we have, by \cref{quasiconvexitycorollary},
    \[
        \Drift_\mu(\tray{q}{s})\le K\max\left(\Drift_\mu(\tray{q}{0},\Drift_\mu(\tray{q}{t}))\right)= K\Drift_\mu(\tray{q}{t})
    \]
\end{proof}

\cref{le:driftgrowth}, combined with \cref{pr:uniformboundthickpart}, gives upper and lower bound on the growth of the drift along a geodesic provided said geodesic does not spend too much continued time outside the thick part. For a given geodesic $\tray{q}{\cdot}$, $\epsilon>0$ and $t>0$, we aim to find some control on the largest $s^{bot}_q(t)\le t$ and lowest $s^{top}_q(t)\ge t$ such that $\tray{q}{s^{bot}_q(t)},\tray{q}{s^{top}_q(t)}\in \T_\epsilon(S)$. Given a geodesic ray $\gamma:\R_+\to \T(S)$, define the proportion of the amount of time spent in the thick part up to time $t$ as
\[
	\thick_\epsilon^{\%}(\gamma,t):=\frac{|\{0 \le s \le t\,:\,\gamma(s) \in\T_\epsilon(S)\}|}{t}.
\]

\begin{theorem} \label{th:exponentialgrowthprecise}
Let $o\in \T(S)$ be a basepoint in \teichmuller space. Furthermore, let $\sigma$ be a measure on $Q^1(o)$ such that for all $0<\xi<1$ there is some $\epsilon(\xi)>0$ such that for $\sigma$-almost every $q$ there is $t_q^\xi<\infty$ such that
\[
\thick_{\epsilon(\xi)}^{\%}(\tray{q}{\cdot},t)\ge \xi
\]
for all $t\ge t^\xi_q$. 

Fix then $0<\theta<1$. For $\sigma$-almost all directions $q\in Q^1(o)$ there is $t(q,\theta)<\infty$ such that
\[
e^{\theta t}< \Drift_\mu(\tray{q}{t})< e^{\frac{1}{\theta}t}
\]
for all $t>t(q,\theta)$.
\end{theorem}

\begin{proof}
Given $0<\theta<1$, let $\xi = \frac{1+\theta}{2}$. For a given $q\in Q^1(o)$ and $t>0$ let $s^{top}_q(t)\ge t$ be the smallest time larger than $t$ such that $\tray{q}{s^{top}_q(t)}\in \T_{\epsilon(\xi)}(S)$. The time spent outside $\T_{\epsilon(\xi)}(S)$ directly after $t$ is $s^{top}_q(t)-t$. Hence, $\frac{s^{top}_q(t)-(s^{top}_q(t)-t)}{s^{top}_q(t)}>\xi$. Therefore, $s^{top}_q(t)<\frac{1}{\xi}t$. Hence, 
 By \cref{le:driftgrowth} we have $K,t_o>0$ such that, for all $t>t_o$,
\[
\Drift_\mu(\tray{q}{t})\le K\Drift_\mu(\tray{q}{s^{top}_q(t)}).
\]
Since $\tray{q}{s^{top}_q(t)}$ is in the $\epsilon$-thick part of \teichmuller space we have, by \cref{pr:uniformboundthickpart},
\[
\Drift_\mu(\tray{q}{t})\le K\Drift_\mu(\tray{q}{s^{top}_q(t)}\le CK e^{s^{top}_q(t)}\le CK e^{\frac{1}{\xi} t}.
\]
Similarly, denoting $s^{bot}_q(t)$ the largest time smaller than $t$ such that $\tray{q}{s^{top}_q(t)}\in \T_{\epsilon(\xi)}(S)$ we get $s^{bot}_q(t)>\xi t$. Following the same reasoning we get 
\[
\Drift_\mu(\tray{q}{t})\ge \frac{1}{K} \Drift_\mu(\tray{q}{s^{bot}_q(t)}\ge \frac{c}{K} e^{s^{bot}_q(t)}\ge \frac{c}{K} e^{\xi t}.
\]
Since $\xi<\theta<1$ there is some $t_\theta$ such that $e^{\theta t}>\frac{c}{K} e^{\xi t}$ and $e^{\frac{1}{\theta} t}<CK e^{\frac{1}{\xi} t}$ for all $t\ge t_\theta$. Hence, the theorem follows from setting $t(q,\theta)=\max(t_q,t_\theta,t_o)$.
\end{proof}
It follows from Downdall--Duchin--Masur \cite[Proposition 5.5]{DowdallStatisticalHyp} that the hypothesis of \cref{th:exponentialgrowthprecise} is satisfied a wide variety of Lebesgue-class measures, giving us a proof of \cref{th:exponentialGrowthalmosteverywhere}. Furthermore, from Azemar--Gadre--Jeffreys \cite[Proposition 3.3]{AzemarStatisticalHyp} the property is also satisfied for harmonic measures generated by non-elementary measures on the mapping class group with finite first moment.

\begin{corollary}\label{co:limits}
Let $\sigma$ be a measure on $Q^1(o)$ satisfying the hypothesis of \cref{th:exponentialgrowthprecise}. Then, for any $\epsilon>0$ we have, for $\sigma$-almost every  $q\in Q^1(o)$,
\[
\liminf_{t\to \infty} \frac{\Drift_\mu(\tray{q}{t})}{e^{(1-\epsilon) t}}=\infty
\]
and
\[
\limsup_{t\to \infty} \frac{\Drift_\mu(\tray{q}{t})}{e^{(1+\epsilon) t}}=0.
\]
Furthermore, there is $\sigma$-almost surely some $0<c(q)<\infty$ such that
\[
\liminf_{t\to \infty} \frac{\Drift_\mu(\tray{q}{t})}{e^{t}}\le c(q) \le \limsup_{t\to\infty} \frac{\Drift_\mu(\tray{q}{t})}{e^t}.
\]
\end{corollary}
\begin{proof}
Let $\theta=1-\epsilon$. Then, by \cref{th:exponentialGrowthalmosteverywhere} there is some $c>0$ such that $\sigma$-almost every $q\in Q^1(o)$ we have
\[
\Drift_\mu(\tray{q}{t})\ge e^{(1-\epsilon) t}
\]
for all $t$ big enough. Hence, the first relation follows. Similarly for $\theta=\frac{1}{1+\epsilon}$ we get the second relation. The last relation follows from taking a diverging sequence of times $(t_n)$ such that $\tray{q}{t_n}\in \T_\epsilon(S)$ and applying \cref{pr:uniformboundthickpart}.
\end{proof}

\section{Variable growth}
In this section we prove that there is some geodesic along which we have variable growth. The basic idea of the proof is alternating \cref{pr:uniformboundthickpart} and \cref{th:linearalongJenkin}, using the fact that both results apply to dense sets of directions.
 \begin{proof}[Proof of \cref{th:undefinedGrowth}]
Fix $\delta,\epsilon>0$ such that the basepoint $o\in \T(S)$ is in the $\epsilon$-thick part of \teichmuller space.
We shall build inductively a sequence of quadratic differentials $(j_k)$, as well as sequences of times $(s_k)$, $(t_{k})$ such that, $\Drift_\mu(\tray{j_k}{s_i})<\left(f(s_i)-\delta2^{-2(k-i)}\right) s_i$ for each $i\le k$ and each $\tray{j_k}{t_i}$ is at most at distance $\delta\left(2-2^{-2(k-i)}\right)$ from $\T_\epsilon$ for each $i\le k$. The theorem will follow by taking an accumulation point of such sequence.

For $k=0$ let $j_0\in Q^1(o)$ be a Jenkins-Strebel quadratic differential. By \cref{th:linearalongJenkin} there are constants $C(j_0),t(j_0)$ such that $\Drift_\mu(\tray{j_0}{t}) < C(j_0) t$ for all $t>t(j_0)$. Let $s_0$ be big enough such that $f(s_0)> C(j_0)+\delta$ and $s_0>t(j_0)$. Finally, let $t_0=0$.

Assume then we have the sequence up to $k$. The set of recurrent directions to $\T_\epsilon$ is dense, so we can take a sequence of quadratic differentials $(q^n)\subset Q^1(o)$ spawning recurrent geodesics and converging to $j_k$. Since $q^n\to j_k$, the geodesics $\tray{q^n}{\cdot}$ converge to the geodesic $\tray{j_k}{\cdot}$ pointwise. The drift is a continuous function with respect to \teichmuller space, so $\Drift_\mu(\tray{q^n}{s_i})\to \Drift_\mu(\tray{j_k}{s_i})\le(f(s_i)-\delta2^{-2(k-i)}) s_i$ for each $i\le k$. Let $q$ be the first element of the sequence $(q^n)$ such that 
\[\Drift_\mu(\tray{q}{s_i})< (f(s_i)-\delta2^{-(2(k-i)+1)}) s_i\]
and 
\[d_{\teich}(\tray{q}{t_i},\tray{j_k}{t_i})<\delta 2^{-(2(k-i)+1)}\]
for all $i\le k$. The geodesic $\tray{q}{\cdot}$ is recurrent, so we can fix $t_{k+1}>s_k+1$ such that $\tray{q}{t_{k+1}}\in \T_\epsilon$.

The set of Jenkins--Strebel directions is dense, so we can take a sequence $(j^n)$ converging to $q$. As before, the convergence within the sequences is pointwise, so $\Drift_\mu(\tray{j^n}{s_i})\to \Drift_\mu(\tray{q}{s_i})\le(f(s_i)-\delta2^{-(2(k-i)+1)}) s_i$ for each $i\le k$. Let $j_{k+1}$ be the first element of the sequence $(j^n)$ such that 
\[\Drift_\mu(\tray{j_{k+1}}{s_i})\le(f(s_i)-\delta2^{-2(k+1-i)}) s_i\]
and 
\[d_{\teich}(\tray{j_{k+1}}{t_i},\tray{q}{t_i})<\delta2^{-2(k+1-i)}\]
for all $i\le k+1$. As before, there is some $C(j_{k+1})$ such that $\Drift_\mu(\tray{j_{k+1}}{t}) < C(j_{k+1}) t$, so let $s_{k+1}$ be the first time larger than $t_{k+1}$ such that $f(s_{k+1})>C(s_{k+1})+\delta$. Furthermore, for $i\le k$
\[
d_{\teich}(\tray{j_{k+1}}{t_i},\T_\epsilon)\le d_{\teich}(\tray{j_{k+1}}{t_i},\tray{q}{t_i})+d_{\teich}(\tray{q}{t_i},\tray{j_k}{t_i})+d_{\teich}(\tray{j_k}{t_i},\T_\epsilon)
\]
\[<\delta \left(2^{-(2(k-i)+1)}+ 2^{-2(k+1-i)}+2-2^{-(2(k-i))}\right)=\delta\left(2-2^{-2(k+1-i)}\right)
\]
and for $i=k+1$ we have, since $\tray{q_{k+1}}{t_{k+1}}\in \T_\epsilon$,
\[d_{\teich}(\tray{j_{k+1}}{t_{k+1}},\T_\epsilon) \le d_{\teich}(\tray{j_{k+1}}{t_{k+1}},\tray{q_{k+1}}{t_{k+1}})< \delta.
\]
Hence, we have completed the induction step.

Let $q_f$ be an accumulation point of the sequence $(j_k)$. There is then a subsequence, relabeled $(j_{k})$ converging to $q_f$. By pointwise convergence, $\Drift_\mu(\tray{j_k}{s_i})\to \Drift_\mu(\tray{q_f}{s_i}$, so $\Drift_\mu(\tray{q_f}{s_i}< f(s_i) s_i$ for each $i$. Furthermore, $\tray{j_k}{t_i}\to \tray{q_f}{t_i}$, so for each $i$ the points $\tray{q_f}{t_i}$ are at most at distance $2\delta$ from the $\T_\epsilon$. Therefore, there is some $\epsilon'$ such that  $\tray{q_f}{t_i}\in T_{\epsilon'}$ for all $i$. Hence, by \cref{pr:uniformboundthickpart} there is some $c>0$ such that $\Drift_\mu(\tray{q_f}{t_i})> c e^{t_i}$. Furthermore, $t_{k+1}> s_k+1> t_{k}+1$, so the sequence $(t_k)$ diverges to infinity, and so does $(s_k)$. Finally, the geodesic $\tray{q_f}{\cdot}$ is recurrent, so by Masur's criterion \cite[Theorem 1.1]{MasurRecurrent} the vertical foliation of $q_f$ is uniquely ergodic.
 \end{proof}

\section{Conclusions}

\cref{th:exponentialgrowthprecise} shows that along a typical geodesic the drift grows exponentially. However, it does not determine precisely the fluctuations within such exponential growth. Furthermore, \cref{th:undefinedGrowth} shows that the behaviour of the drift along a geodesic can vary wildly, so it is natural to ask whether there is some variation within a typical geodesic.
\begin{question}
Let $\nu$ be a measure on $Q^1(o)$ satisfying the hypothesis of \cref{th:exponentialgrowthprecise}. Do we have
\[
0<\liminf_{t\to \infty} \frac{\Drift_\mu(\tray{q}{t})}{e^t}=\limsup_{t\to\infty} \frac{\Drift_\mu(\tray{q}{t})}{e^t}<\infty
\]
$\nu$-almost surely?
\end{question}
Note that by \cref{co:limits} we have $\limsup_{t\to\infty} \frac{\Drift_\mu(\tray{q}{t})}{e^t}>0$ and $\liminf_{t\to \infty} \frac{\Drift_\mu(\tray{q}{t})}{e^t}<\infty$. In the proof of \cref{th:exponentialgrowthprecise} we have bounded the quotient between the drift and the exponential $e^t$ by a function depending on the maximal continuous time spent in the thin part of \teichmuller space up to time $t$. The growth of these maximal departures may vary differently depending on the measure. On the one hand, following work of Gadre \cite[Lemma 5.5]{VaibhavExcursions} it is reasonable to conjecture that the maximal departure grows slightly faster than $\log(t)$ for the Lebesgue measure. On the other hand, because of exponential decay of subsurface projections for harmonic measures, it may be expected (though this is still unproved) that the largest continuous time spent in the thin part for a harmonic typical \teichmuller geodesic is of the order of $\log( \log( t))$. Therefore, the answer to the previous question might be different for the Lebesgue and harmonic measures.

\bibliographystyle{alpha}
\bibliography{bibliography.bib}{}


\end{document}